\documentclass[12pt]{article}
      \usepackage{setspace}     
     \usepackage[pdftex]{graphicx}
     \usepackage{sidecap}
\usepackage{fancyhdr, amsmath, amssymb, amsfonts, url, dsfont, mathrsfs, graphicx, epsfig,  amsthm,mathrsfs}
\usepackage{tikz}
\usepackage{wrapfig}
\usepackage{framed}
\usepackage{listings}

\usepackage{pgf}
\usepackage{tikz}
\usetikzlibrary{arrows,automata}

\usepackage{dsfont}
\usepackage{float}
\usepackage{verbatim} %para los comentarios por ejemplo
\usepackage[latin1]{inputenc}
\usepackage{cite}
\usepackage{graphicx}
\usepackage{vmargin} % Supongo que tiene que ver con el margen
\usepackage{ifpdf} % Para ingresar metadatos e indices en el PDF
\usepackage{url} % URL's bonitas
\usepackage{fancyhdr}
\usepackage{lmodern}
\usepackage{moreverb}
\usepackage{enumitem}
\usepackage{amsmath}
\usepackage{pgf}
\usepackage{tikz}
\usepackage{graphicx}
\usetikzlibrary{arrows}
\usetikzlibrary{shapes,snakes,calendar,matrix,backgrounds,folding}

\usepackage{color}
\definecolor{rltred}{rgb}{0.75,0,0}
\definecolor{rltgreen}{rgb}{0,0.5,0}
\definecolor{rltblue}{rgb}{0,0,0.75}

\usepackage{thumbpdf}
\usepackage[pdftex,colorlinks=true,urlcolor=rltred, filecolor=rltgreen, linkcolor=rltblue]{hyperref}

      \usepackage{caption}
      \usepackage{subcaption}
      \usepackage{breqn}	
      \theoremstyle{plain}
      \newtheorem{theorem}{Theorem}[section]

      \newtheorem{prop}[theorem]{Proposition}

      \newtheorem{definition}[theorem]{Definition}

      \newcommand{\R}{{\mathbb R}}

      \newcommand{\B}{\mathcal{B}}
      \newcommand{\Z}{\mathbb{Z}}
      \newcommand{\N}{\mathbb{N}}

      \newcommand{\sI}{\mathcal{I}}

    \newcommand{\sU}{\mathcal{U}}

      \newcommand{\sX}{\mathcal{X}}

\begin{document}

\title{A Large deviation and an escape rate result for special semi-flows}

\author{Italo Cipriano}
\date{December 2015}

\maketitle

\begin{abstract} In this paper we consider a smooth flow $(\Lambda,\Phi^t)$ builded from suspending over a (non-invertible topologically mixing) subshift of finite type, and we equip it with an equilibrium measure $\nu$ on $\Lambda.$ The two main theorems are a large deviation and an escape rate result.
The first theorem gives an explicit formula for $X>0$ and $Y$ such that $$\nu\left\{x\in\Lambda: \left|\int F\circ \Phi^s (x) ds-\int F d\mu\right|>\epsilon\right\}\leq \exp(-Xt+\log t+Y)$$ for $t\gg>1\gg\epsilon>0,$ where $F:\Lambda\to\R$ is smooth. The second theorem gives an explicit lower bound for the asymptotic behaviour of the escape rate of $\nu$ through a small hole.
\end{abstract}
  
  \section{Introduction}\label{ERLD_intro}

In this paper we study two phenomena, large deviations and escape rates for uniformly hyperbolic smooth semi-flow. We consider the case of special semi-flows over subshifts of finite type, that in many occasions constitute a first step toward the study of hyperbolic flows. Our proofs use the machinery available of thermodynamic formalism for discrete dynamical systems, in particular  we require the measure to be an equilibrium state, so that we can extend the results for the dynamics of the flow. To make this precise, 
suppose that we have a measure preserving (discrete) dynamical system $(\sX,\B,\sigma,\mu),$ where $(\sX,\sigma)$ is a topologically mixing subshift of finite type, $\B$ the Borel algebra on $\sX$ and $\mu$ an equilibrium state of H\"older potential. That is, an invariant probability measure $\mu$ that achieves the supremum
$$
h_{\mu}(\sX)+\int \varphi d\mu
$$
among all the invariant probabilities measures on $\sX,$ where $h_{\mu}(\sX)$ is the measure theoretic entropy and $\varphi:\sX\to\R$ is a H\"older potential. We define $P=P(\varphi):=h_\mu(\sX)+\int\varphi d\mu,$ when $\mu$ is an equilibrium state of $\varphi.$ A special semi-flow $(\Lambda, \Phi^t)$ over $(\mathcal{X},\sigma),$ corresponds to the semi-flow in which every point in $\Lambda$ moves with unit speed along the non-expanding direction until it reaches the boundary of $\Lambda$ and it jumps according $\sigma.$ That is, for a continuous function $f:\mathcal{X}\to\R^{>0},$ we consider the continuous action $\Phi^t$ 
  on
  \[
  \Lambda:=\{(x,t):x\in\sX,0\leq t<f(x)\}\subset \sX\times\R^{\geq 0}
  \]
 onto itself defined by  
\[ 
     \Phi^{t}(x,s):= \left(\sigma^{m}x,s+t- \sum_{k=0}^{m-1} f(\sigma^k x)\right)\mbox{ for } \sum_{k=0}^{m-1} f(\sigma^k x)\leq s+t<\sum_{k=0}^{m} f(\sigma^k x),
\]
where $m\in \Z^{\geq 0}.$ On $\Lambda$ we consider the invariant and ergodic probability measure $$\nu=\frac{\mu\times Leb}{\int f d\mu},$$ where $Leb$ is the Lebesgue measure.

Large deviations estimate the asymptotic measure of the bad points for the pointwise Birkhoff's theorem, i.e., given an observable $F:\Lambda\to\R,$ studies  
\[
Z(t)=Z(\epsilon,\nu,f,F,t):=\nu\left\{(x,l)\in \Lambda: \left| \frac{1}{t}\int_0^t F \circ \Phi^s(x)ds-\int F d\mu\right |>\epsilon\right\}
\] 
for $t\gg 1\gg\epsilon>0.$ Results about large deviations for uniformly hyperbolic dynamical systems can be found in \cite{young90,CM2011} and references therein. From an historic point of view, large deviations in dynamical systems have been studied in the direction of generalizing (or finding a similar phenomema) for a larger family of dynamics. Some papers in this direction are \cite{7.10.2015_10.43,7.10.2015_10.47,7.10.2015_10.36,7.10.2015_10.46,7.10.2015_10.50}. Our first result follows other direction, we consider dynamics for which large deviation results are well known and we relate the parameter of the dynamics with the parameters of the large deviation, improving the estimation of a well known result. Let us make this explicit. It is well known, see for example\cite{young90,kifer_19}, that for $F:\Lambda\to\R$ H\"older we have that for every $\epsilon>0$ 
\begin{equation}\label{well_known}
\limsup\limits_{t\to +\infty}  \frac{\log Z(t) }{t}<0. 
\end{equation}

Our first result improves the estimate of the asymptotic behaviour of $Z(t)$ as $t\to +\infty$  given by (\ref{well_known}).

\begin{theorem}\label{teo_uno} If  $F:\Lambda\to\R$ is a Lipschitz function. Then, 
$$ Z(t)\leq \exp(-Xt+\log t +Y)$$
for $t\gg 1\gg\epsilon>0.$ Where $X=X(\varphi,f,F,\epsilon)\in\R^{>0}$ and  $Y=Y(\varphi,f,F,\epsilon)\in\R$ have an explicit formula.
\end{theorem}

The proof follows from a result in \cite{CM2011} and standard arguments, see \cite{Melbourne}, Section 5, in particular the arguments in proofs of Theorem 5.1 and 5.3.\\

Our second theorem deals with escape rates for $(\Lambda,\Phi^t,\nu).$ In this setting, escape rates study the limit
$$
\lim_{\nu(\sU)\to 0}\frac{R(\sU)}{\nu(\sU)},
$$
where
$$
R(\sU):=-\limsup_{t\to+\infty}\frac{1}{t}\log\nu\left\{(x,l)\in \Lambda: \bigcup_{0<s<t}\Phi^s(x,l)\cap \sU=\emptyset\right\}.
$$
This limit quantifies the asymptotic rate (as $t$ goes to infinity) of the measure of the points $(x,l)\in \Lambda$ that has not enterered to a subset $\sU\subset \Lambda$ until time $t,$ with respect to the measure of $\sU,$ when $\sU$ is small. Escape rates for discrete dynamical systems are studied in \cite{FP} and in the references therein.
 Our second theorem gives a lower bound for the escape rates that depend on the shrinking sequence and on $f.$
  
  \begin{theorem}\label{teo_dos}
If the roof function $f:\sX\to \R^{>1}$ is Lipschitz, and $\{\sI_n\},  \sI_n\subset \sX$ is a sequence of open sets that satisfies the nested condition (Definition \ref{nested condition}) with $\cap_{n\in\N}\sI_n=\{z\}$ for $z\in\sX.$ Then  
$$
\lim_{n\to\infty}\frac{{R}(\sI_n\times\{0\})}{\nu(\sI_{n}\times[0,1])}\geq 
\begin{cases}
\frac{1}{W} &\mbox{ if } z \mbox{ is not periodic }\cr
\frac{1-\exp{\left(\sum_{k=0}^{p-1}\varphi(\sigma^{k}x)-p P(\varphi)\right)}}{W} &\mbox{ if } z \mbox{ has prime period }p,
\end{cases}
$$
where $W=W(f)\in \R^{\geq 2}$ has an explicit formula.
%where $D=1 + \frac{\left\Vert f\right\Vert}{\int f d\mu}.$
\end{theorem}
    
This results is a weaker version of the main theorem in \cite{ERSF_DSS}. However, our demonstration here is more direct and avoids discretizing the flow.\\
    
We write this paper in three main sections. In the first, we precisely define our framework, in the second we provide the proof of Theorem \ref{teo_uno} and in the third we provide the proof of Theorem \ref{teo_dos}.
    
\section{Framework}
  
In this section we define subshifts of finite type, the spaces of continuos and Lipchitz functions with their respective norms and the  nested condition. At the end, we state the main results behind our proofs.\\

Let $A$ denote an irreducible and aperiodic $a\times a$ matrix of zeros and ones ($a\in\Z^{\geq 2}$), i.e. there exists $d\in\N$ for which $A^{d}>0$ (all coordinates of $A^d$ are strictly positive). We call the matrix $A$ transition matrix. We define the non-invertible topologically mixing subshift of finite type $\mathcal{X}=\mathcal{X}_A \subset \{ 1,\ldots,a\}^{\Z^{\geq 0}}$ such that
  $$\mathcal{X}:=\{(x_n)_{n=0}^{\infty}:A(x_n,x_{n+1})=1\mbox{ for all }n\in \Z^{\geq 0}\}.$$ 
  On $\mathcal{X},$ the shift $\sigma:\mathcal{X}\to \mathcal{X}$ is defined by $\sigma(x)_{n}=x_{n+1}$ for all $n\in \Z^{\geq 0}.$ For $x\in \mathcal{X}$ and $n\in \N,$ we define the cylinder $$[x]_n:=\{y\in\mathcal{X}: y_i=x_i \mbox{ for } i\in\langle 0, n-1\rangle\},$$ 
      we denote by $\xi_n$ the set of all the cylinders $[x]_n$ with $x\in\mathcal{X}.$ Given $\theta\in(0,1),$ we consider the metric on $\mathcal{X}$ given by $d_{\theta}(x,y)=\theta^{m},$ where $m=\inf\{n\in \N: x_n\neq y_n\}$ and $d(x,x)=0$ for every $x\in \mathcal{X}.$ Here $(\mathcal{X}, d_{\theta})$ is a complete metric space. We say that $f:\mathcal{X}\to\R$ is continuous if it is continuous with respect to $d_{\theta}.$ Given $f:\mathcal{X}\to\R$ continuous and $n \in \N$ define 
 $$
 \begin{aligned}
  S^{\sigma}_n f (\cdot)&:=\sum_{k=0}^{n-1}f(\sigma^k \cdot),\\
 V_n(f)&:= \sup_{z\in \mathcal{X}}\{|f(x)-f(y)|:x,y\in [z]_n\},
 \end{aligned}
$$
the Lipschitz semi-norm
\[
|f|_{\theta}:=\sup\left\{ \frac{V_n(f)}{\theta^n}:n \in\N\right\}
\]    
and the Lipschitz norm
\[
\left\Vert f\right\Vert_{\theta}:= |f|_{\theta}+\| f \|,
\]
where $\| f \|:=\sup_{x\in \mathcal{X}} \{|f(x)|\}.$ The space of continuous functions with finite Lipschitz norm is called the space of Lipschitz functions (or $\theta$-Lipschitz functions) and denoted by $\mathcal{F}.$ A continuous function is $\alpha$-H\"older for $d_{\theta}$ if and only if it is Lipschitz for $d_{\theta^{\alpha}}.$ Recall that given a H\"older potential $\varphi\in \mathcal{F},$ there is unique equilibrium state ( see \cite{Rufus}).

The nested condition is a technical condition used in \cite{FP} that we define in what follows.

\begin{definition}[Nested condition]\label{nested condition}   
  We say that a family of open sets $\{\mathcal{I}_n\},\mathcal{I}_n\subset \mathcal{X}$  satisfies the \emph{nested condition} if it satisfies that:
  \begin{enumerate}
  \item each $\mathcal{I}_n$ consists of a finite union of cylinder sets, with each cylinder having length $n;$
  \item $\mathcal{I}_{n+1}\subset \mathcal{I}_n$ for every $n\in \N$ and $\cap_{n\in\N}\mathcal{I}_n=\{z\}$ for some $z\in \mathcal{X};$
  \item there exist constants $c\in \R^{>0}$ and $0<\rho<1$ such that $\mu(\mathcal{I}_n)\leq c\rho^{n}$ for all $n\in \N;$
  \item there is a sequence $\{l_n\}\subset \N$ and a constant $\kappa \in\R^{>0}$ such that $\kappa<l_n/n\leq 1$ and $\mathcal{I}_n\subset [z]_{l_n}$ for all $n\in \N;$
  \item if $\sigma^{p}(z)=z$ has prime period $p,$ then $\sigma^{-p}(\mathcal{I}_n)\cap [z]_p\subset \mathcal{I}_n$ for large enough $n.$
  \end{enumerate}
\end{definition}

The main tools that we use to prove our results are the following two theorems.

\begin{theorem}[Theorem 5.1 in \cite{FP}]\label{teoFP} For shrinking sequences  $\{\mathcal I_n\},\mathcal I_n\subset \mathcal X$ satisfying the nested condition with $\cap_{n\in\N}\mathcal I_n=\{z\},z\in \mathcal X$ 
$$
\lim_{n\to\infty} \frac{R_{\text{Discrete}}(\mathcal{I}_n)}{\mu(\mathcal{I}_n)}=\gamma(z)
$$
where 
$$
\gamma(z):=\begin{cases}
1&\mbox{ if }z\mbox{ is not periodic,}\\
1-\exp{\left(\sum_{k=0}^{p-1}\varphi(\sigma^{k}x)-p P(\varphi)\right)}&\mbox{ if }z\mbox{ has prime period }p
\end{cases}
$$
 and 
$$ 
R_{\text{Discrete}}(\mathcal{I}_n):= -\limsup_{k\to\infty}\frac{1}{k} \log \mu \{x\in\mathcal X: \sigma^ix\notin \mathcal \mathcal{I}_n, i\in \{0,1,\ldots, k-1\} \}.
$$
\end{theorem} 

\begin{theorem}[Corollary 3.3 in \cite{CM2011}]\label{cm2011}
  Let $g:\mathcal{X}\to \R$ be $\theta$-Lipschitz and $\mu$ be the equilibrium state of a H\"older potential $\varphi$. Then
  \[
  \mu\left\{ x: \left| \frac{1}{m} S^{\sigma}_m g (x)-\int g d\mu \right|\geq \epsilon \right\}\leq 2e^{-Bm\epsilon^2}
  \]
  for every $\epsilon\in\R^{>0}$ and for every $m\in\N,$ where $B:=(4D|g|_{\theta}^2)^{-1}$ and $D=D(\varphi)$ is a constant independent of $g.$
\end{theorem}  

Consider a special flow $(\Lambda,\Phi^t)$ over a subshift of finite type $(\sX,\sigma)$ with roof function $f$ and let $\theta\in (0,1).$ We will define what we understand by $F:\Lambda\to\R$ to be Lipschitz.

\begin{definition}
 Define $\tau(x,t)=\min\{s\in\R^{>0}:\Phi^s(x,t)\in\sX\times\{0\}\}$ for $(x,t)\in \Lambda.$ The space $\Lambda$ is a metric space with the metric 
$$
d_{\Lambda}((x,t),(y,s)):=\min \begin{cases}
d_{\theta}(x,y)+|t-s|,\\
d_{\theta}(\sigma x,y)+\tau(x,t)-t+s,\\
d_{\theta}(x,\sigma y)+\tau(y,s)-s+t.\\
\end{cases}
$$
We say that a map $F:\Lambda\to \R$ is Lipschitz, if it is Lipschitz with respect to the metrics $d_{\Lambda}$ on $\Lambda$  and  to $d_{\R}(x,y):=|x-y|$ on $\R.$ Given a $(d_{\Lambda},d_{\R})$-continuous function $F:\Lambda\to\R,$ we define $$\|F\|=\sup_{x\in\sX}\sup_{s\in [0,f(x))} \{|F(x,s)|\}.$$ 
\end{definition}

In this paper we will require $F:\Lambda\to \R$ to satisfies a weaker condition than be Lipschitz, indeed we need that there exists $C\in\R^{>0}$ such that for every $x,y\in\sX$ 
\begin{equation}\label{ineq:24072015:3:05}
\int_{0}^{\min \left( f(x),f(y) \right)}|F(x,s)-F(y,s)|ds\leq C d_{\theta}(x,y).
\end{equation}

\section{Proof of Theorem \ref{teo_uno}}

Along this section consider a special flow $(\Lambda,\Phi^t)$ over a subshift of finite type $(\sX,\sigma)$ with $\theta$-Lipschitz roof function $f:\sX\to\R^{\geq 1},$ for some $\theta\in (0,1).$ We consider $\mu$ an equilibrium state of H\"older potential on $\sX$ and the invariant and ergodic probability measure $\nu=\frac{\mu\times Leb}{\int f d\mu}$ on $\Lambda.$

The key proposition we will use to prove Theorem \ref{teo_uno} is the following.

\begin{prop}\label{CMSCM}
If $F:\Lambda\to\R$ satisfies (\ref{ineq:24072015:3:05}),
then there are constants $C_1,C_2\in\R^{>0}$ depending on $f$ and $F$ such that for all $\epsilon\in\R^{>0},$ for all $t\in\R^{>\max\left(\frac{\left\Vert f\right\Vert \left\Vert F\right\Vert(1+\left\Vert f\right\Vert)}{\epsilon}, 2\left\Vert f\right\Vert \right)},$
$$
\begin{aligned}
&\mu\left\{x\in\sX: \left|  \frac{1}{t} \int_0^t F\circ \Phi^s(x,0)ds-\int F d\nu \right| \geq \epsilon \right\}\cr
&\leq 2t\left\Vert f\right\Vert\exp \left\{ -C_1 \left( \frac{t}{\left\Vert f\right\Vert}-2\right) \left( \epsilon- \frac{\left\Vert f\right\Vert \left\Vert F\right\Vert}{t}(1+\left\Vert f\right\Vert) \right)^2 \right\}\cr
&\quad+2t\left\Vert f\right\Vert\exp \left\{ -C_2 \left( \frac{t}{\left\Vert f\right\Vert}-2\right) \left( \epsilon- \frac{\left\Vert f\right\Vert \left\Vert F\right\Vert}{t}(1+\left\Vert f\right\Vert) \right)^2/(\left\Vert f\right\Vert \left\Vert F\right\Vert)^2 \right\}.
\end{aligned}
$$
\end{prop}

%We now proceed to prove Proposition \ref{CMSCM}.

\begin{proof}
Suppose $t>\left\Vert f\right\Vert$ and define $\tilde{F}:\sX\to \R,x\mapsto \int_{0}^{f(x)}F(x,s)ds.$ Given $x\in \sX$ we can write $t=S^{\sigma}_{n(x)} f(x)+t(x)$ for some $n(x)\in\N$ and $f(\sigma^n x)>t(x)\geq 0,$ then $n(x)\leq t=S^{\sigma}_{n(x)} f(x)+t(x)\leq (n(x)+1) \left\Vert f\right\Vert.$ In particular, $t\geq n(x)\geq \frac{t}{\left\Vert f\right\Vert}-1.$  Keeping this in mind we have the following inequalities: 

$$
\begin{aligned}
&\mu\left\{x\in\sX: \left|  \frac{1}{t} \int_0^t F\circ \Phi^s(x,0)ds-\int F d\nu \right| \geq \epsilon \right\}\cr
&=\mu\left\{x\in\sX: \left|  \frac{S^{\sigma}_{n(x)} \tilde{F}(x)+\int_{0}^{t(x)}F(\sigma^{n(x)}x,s)ds}{S^{\sigma}_{n(x)}f(x)+t(x)} -\frac{\int \tilde{F} d\mu}{\int f d\mu} \right| \geq \epsilon \right\}\cr
&\leq\mu\left\{x\in\sX: \left|  \frac{S^{\sigma}_{n(x)} \tilde{F}(x)}{S^{\sigma}_{n(x)}f(x)} \frac{S^{\sigma}_{n(x)}f(x)}{S^{\sigma}_{n(x)}f(x)+t(x)}-\frac{\int \tilde{F} d\mu}{\int f d\mu} +\frac{\int_0^{t(x)}F(\sigma^{n(x)} x,s)ds}{t}  \right| \geq \epsilon \right\}\cr
&\leq\mu\left\{x\in\sX: \left|  \frac{S^{\sigma}_{n(x)} \tilde{F}(x)}{S^{\sigma}_{n(x)}f(x)} -\frac{t(x) }{t} \frac{S^{\sigma}_{n(x)} \tilde{F}(x)}{S^{\sigma}_{n(x)}f(x)}-\frac{\int \tilde{F} d\mu}{\int f d\mu} \right| + \frac{\left\Vert f\right\Vert \left\Vert F\right\Vert}{t}\geq \epsilon \right\}\cr
&\leq\mu\left\{x\in\sX: \left|  \frac{S^{\sigma}_{n(x)} \tilde{F}(x)}{S^{\sigma}_{n(x)}f(x)} -\frac{\int \tilde{F} d\mu}{\int f d\mu} \right| +\frac{ \left\Vert f\right\Vert^2 \left\Vert F\right\Vert}{t} + \frac{\left\Vert f\right\Vert \left\Vert F\right\Vert}{t}\geq \epsilon \right\}=:(\star),
\end{aligned}
$$
where \[\epsilon_1:= \epsilon -\frac{\left\Vert f\right\Vert \left\Vert F\right\Vert}{t}(1+\left\Vert f\right\Vert).\]
Furthermore, 
$$
\begin{aligned}
&(\star)\cr
&=\mu\left\{x\in\sX: \left|  \frac{S^{\sigma}_{n(x)} \tilde{F}(x)}{n(x)} \frac{n(x)}{S^{\sigma}_{n(x)}f(x)}-\int \tilde{F} d\mu \frac{n(x)}{S^{\sigma}_{n(x)}f(x)}+\int \tilde{F} d\mu \frac{n(x)}{S^{\sigma}_{n(x)}f(x)} -\frac{\int \tilde{F} d\mu}{\int f d\mu} \right| \geq \epsilon_1 \right\}\cr
&\leq\mu\left\{x\in\sX: \frac{n(x)}{S^{\sigma}_{n(x)}f(x)} \left|  \frac{S^{\sigma}_{n(x)} \tilde{F}(x)}{n(x)} -\int \tilde{F} d\mu\right| +\left|\int \tilde{F} d\mu\right| \left|\frac{n(x)}{S^{\sigma}_{n(x)}f(x)} -\frac{1}{\int f d\mu} \right| \geq \epsilon_1 \right\}\cr
&\leq\mu\left\{x\in\sX:  \left|  \frac{S^{\sigma}_{n(x)} \tilde{F}(x)}{n(x)} -\int \tilde{F} d\mu\right| \geq \frac{\epsilon_1}{2}\right\}\cr
&\quad+\mu\left\{x\in\sX:\left|\int \tilde{F} d\mu\right| \left|\frac{S^{\sigma}_{n(x)}f(x)}{n(x)} -\int f d\mu \right| \geq \frac{\epsilon_1}{2} \right\}=:(\star \star),
\end{aligned}
$$
where 
\[\epsilon_2:= \frac{\epsilon_1}{2}, \epsilon_3:= \frac{\epsilon_1}{2\left\Vert f\right\Vert \left\Vert F\right\Vert} \mbox{ and }n_1(t):= \left\lfloor \frac{T}{\left\Vert f\right\Vert }-1\right\rfloor.\] 
Finally, 
$$
\begin{aligned}
(\star \star)&\leq \sum_{n\in\{n_1(t),n_1(t)+1,\ldots,[t]\}}\mu\left\{x\in\sX:  \left|  \frac{S^{\sigma}_{n} \tilde{F}(x)}{n} -\int \tilde{F} d\mu\right| \geq \epsilon_2\right\}\cr
&\quad+\sum_{n\in\{n_1(t),n_1(t)+1,\ldots,[t]\}}\mu\left\{x\in\sX:\left|\frac{S^{\sigma}_{n}f(x)}{n} -\int f d\mu \right| \geq \epsilon_3 \right\}\cr
&\leq 2t\left\Vert f\right\Vert\exp\left(-\frac{\tilde{C_1}}{|\tilde{F}|^2_{\theta}}\cdot n_1(t) \cdot\epsilon_2^2\right)+ 2t\left\Vert f\right\Vert\exp\left(-\frac{\tilde{C_2}}{|f|^2_{\theta}}\cdot n_1(t) \cdot  \epsilon_3^2\right).
\end{aligned}
$$
The map $\tilde{F}$ is $\theta$-Lipschitz, indeed, suppose $x,y\in [z]_m$ for some $z\in\sX,m\in\N,$ and $f(x)>f(y)$ then 
$$
\begin{aligned}
|\tilde{F}(x)-\tilde{F}(y)| &\leq  \int_{f(y)}^{f(y)+|f|_{\theta}\theta^m} \left\Vert F\right\Vert ds+\int_{0}^{f(y)}|F(x,s)-F(y,s)|ds\cr
&\leq (|f|_{\theta} \left\Vert F\right\Vert+ C)\theta^m.
\end{aligned}
$$
Thus, we can write
$$
\begin{aligned}
&\mu\left\{x\in\sX: \left|  \frac{1}{t} \int_0^t F\circ \Phi^s(x,0)ds-\int F d\nu \right| \geq \epsilon \right\}\cr
&\leq 2t\left\Vert f\right\Vert\exp\left(-\frac{1/(4D)}{|\tilde{F}|^2_{\theta}}\cdot \left\lfloor \frac{t}{\left\Vert f\right\Vert }-1\right\rfloor \cdot\left( \frac{\epsilon -\frac{\left\Vert f\right\Vert \left\Vert F\right\Vert}{t}(1+\left\Vert f\right\Vert)}{2}\right)^2\right)\cr
&\quad+2t\left\Vert f\right\Vert\exp\left(-\frac{1/(4D)}{|f|_{\theta}^2}\cdot\left\lfloor \frac{t}{\left\Vert f\right\Vert }-1\right\rfloor \cdot\left( \frac{\epsilon -\frac{\left\Vert f\right\Vert \left\Vert F\right\Vert}{t}(1+\left\Vert f\right\Vert)}{2\left\Vert f\right\Vert \left\Vert F\right\Vert }\right)^2\right),
\end{aligned}
$$
where $D$ is the constant that depends on $\mu$ in \cite{CM2011}, Theorem 3.1.
\end{proof}

We now proceed to the proof of Theorem \ref{teo_uno}.

\begin{proof}[Proof of Theorem \ref{teo_uno}]
If $F:\Lambda\to\R$ is Lipschitz, then it satisfies (\ref{ineq:24072015:3:05}) and we can apply Proposition \ref{CMSCM}. By definition of the probability measure $\nu$ on $\Lambda$ we also have that 
$$
\begin{aligned}
Z(t)&=\nu \left\{(x,l)\in\Lambda: \left|  \frac{1}{t} \int_0^t F\circ \Phi^s(x,l)ds-\int F d\nu \right| \geq \epsilon \right\}\\
&\leq\frac{\|f\|}{\int f d\mu}\mu\left\{x\in\sX: \left|  \frac{1}{t} \int_0^t F\circ \Phi^s(x,0)ds-\int F d\nu \right| \geq \epsilon \right\}.
\end{aligned}
$$

For $t\in \R^{>\max\{2\|f\|,2  \frac{\|f\|\|F\|(1+\|f\|)}{\epsilon}\}}$ by Proposition \ref{CMSCM} we have that 
$$
\begin{aligned}
&\mu\left\{x\in\sX: \left|  \frac{1}{t} \int_0^t F\circ \Phi^s(x,0)ds-\int F d\nu \right| \geq \epsilon \right\}\cr
&\leq 2t\left\Vert f\right\Vert\exp \left\{ -C_1 \left( \frac{t}{\left\Vert f\right\Vert}-2\right) \frac{\epsilon^2}{4} \right\}\cr
&\quad+2t\left\Vert f\right\Vert\exp \left\{ -C_2 \left( \frac{t}{\left\Vert f\right\Vert}-2\right)\frac{\epsilon^2}{4(\left\Vert f\right\Vert \left\Vert F\right\Vert)^2}\right\}.
\end{aligned}
$$
Calling $C=\min\{C_1,C_2\},$ then we have that 
$$
\begin{aligned}
&\mu\left\{x\in\sX: \left|  \frac{1}{t} \int_0^t F\circ \Phi^s(x,0)ds-\int F d\nu \right| \geq \epsilon \right\}\cr
& \leq 4t\| f\|\exp \left\{ -C \left( \frac{t}{\left\Vert f\right\Vert}-2\right)\frac{\epsilon^2}{4(\left\Vert f\right\Vert \left\Vert F\right\Vert)^2}\right\}\\
& = 4t\| f\| \exp \left\{  \frac{2C\epsilon^2}{4(\left\Vert f\right\Vert \left\Vert F\right\Vert)^2} \right\}\exp \left\{\frac{-C\epsilon^2 t}{4\|f\|^3\|F\|^2} \right\}.
\end{aligned}
$$
Calling $Y=\log\left\{4\| f\| \exp \left\{  \frac{2C\epsilon^2}{4(\left\Vert f\right\Vert \left\Vert F\right\Vert)^2} \right\}\right\}$  we have that 
$$
Z(t)\leq \exp\{Y\} t\exp \left\{\frac{-C\epsilon^2 t}{4\|f\|^3\|F\|^2} \right\}
$$
and therefore, calling $X=\frac{C\epsilon^2 }{4\|f\|^3\|F\|^2},$ we have that
$$
\log Z(t)\leq Y+ \log t -Xt 
$$
that concludes the proof.
\end{proof}

\section{Proof of Theorem \ref{teo_dos}}

Along this section consider a special flow $(\Lambda,\Phi^t)$ over a subshift of finite type $(\sX,\sigma)$ with $\theta$-Lipschitz roof function $f:\sX\to\R^{\geq 1},$ for some $\theta\in (0,1).$ We consider $\mu$ an equilibrium state of H\"older potential on $\sX$ and the invariant and ergodic probability measure $\nu=\frac{\mu\times Leb}{\int f d\mu}$ on $\Lambda.$ Finally, let us consider $\{\sI_n\},  \sI_n\subset \sX$ a sequence of open sets that satisfies the nested condition (Definition \ref{nested condition}) with $\cap_{n\in\N}\sI_n=\{z\}$ for $z\in\sX.$ 

  We introduce a definition.
  
  \begin{definition}
  Define for each $n\in\N$ $\tau_{n}:\mathcal{X}\to \N$ by 
  $$
  \tau_n(x):= \inf \{m\in \N: \sigma^m(x)\in \sI_n\}
  $$
  and $\tilde{R}:\{\sI_n\}\to\R$ by
  $$
  \tilde{R}(\sI_n):= -\limsup_{t\rightarrow \infty} \frac{1}{t}\log \mu \{x: S^{\sigma}_{\tau_n(x)}f(x)\geq t\}.
  $$
  \end{definition}

  Clearly, $\tilde{R}(\sI_n)=R(\sI_n)$ for every $n\in\N,$ and so $$\frac{\tilde{R}(\sI_n)}{\mu(\sI_{n})}\int f d\mu=\frac{R(\sI_n)}{\nu(\sI_{n}\times [0,1])}.$$ We will prove here that 
\begin{equation}\label{eq_enough}
\lim_{n\to\infty}\frac{\tilde{R}(\sI_n)}{\mu(\sI_{n})}\geq \frac{\gamma(z)}{\int f d\mu + \left\Vert f\right\Vert},
\end{equation} 
that it is enough to finish the proof of Theorem \ref{teo_dos}, where 
$$W=1+\frac{\|f\|}{\int f d\mu}.$$ 

  \begin{proof}[proof of (\ref{eq_enough})]  
  Fix $0<\epsilon<1/\left\Vert f\right\Vert$ and define $B$ as in Theorem \ref{cm2011}.
 
  We have that 
  \[\frac{-B\epsilon^{3}}{\mu(\sI_n)}\to-\infty \mbox{ as $n$ tends to infinity}\] and 
  \[-\frac{ R_{\text{Discrete}}(\sI_n)}{\mu(\mathcal{U}_{n})}\to-\gamma(z) \mbox{ as $n$ tends to infinity}.\]
 
 Therefore, there exists $n_0\in\N$ and $\mathcal{N}\subset \N$ an infinite set such that for any $n\in\Z^{>n_0}$ and for any $k\in \mathcal{N}$
  \[
  0>\frac{\log \mu\left\{x: \tau_n(x)\geq k\right\}}{ \mu(\sI_n) k \int f d\mu}  >\frac{-B\epsilon^{3}}{\mu(\sI_n)},
  \]
  which implies that 
  \begin{equation}\label{extra1:1:1:1}
  \mu\left\{x: \tau_n(x)\left(\int f d\mu +\epsilon\right)\geq k\right\}>e^{-B\epsilon^{3} k}.
  \end{equation}

We write $\tau_{n}$ instead of $\tau_{n}(x),$ $S^{\sigma}_{\tau_{n}}f$ instead of $S^{\sigma}_{\tau_{n}}f(x)$ and $S^{\sigma}_{s}f$ instead of $S^{\sigma}_{s}f(x)$ when $s\in \Z^{\geq 0}.$ For any $n \in \Z^{>n_0}$ and $[\epsilon t]\in \mathcal{N},$ using inequality (\ref{extra1:1:1:1}) and the identity 
 \begin{equation}\label{24072015_ID1}
 \begin{aligned}
  \mu\left\{ x:S^{\sigma}_{\tau_{n}}f\geq t\right\}	&= \mu\left\{ x:S^{\sigma}_{\tau_{n}}f\geq t,\tau_{n}>\epsilon t, \left|\frac{1}{[\epsilon t]}S^{\sigma}_{[\epsilon t]}f-\int fd\mu\right|<\epsilon\right\}  \\
 &\quad + \mu\left\{ x:S^{\sigma}_{\tau_{n}}f\geq t, \tau_{n}>\epsilon t, \left|\frac{1}{[\epsilon t]}S^{\sigma}_{[\epsilon t]}f-\int fd\mu\right|\geq\epsilon\right\} 
\end{aligned}
 \end{equation}
 we conclude the inequality
 \begin{equation}\label{24072015_INEQ2}
\mu\left\{ x:S^{\sigma}_{\tau_{n}}f\geq t\right\} \leq  \mu\left\{ x:\tau_{n}\left(\epsilon+\int fd\mu + \left\Vert f\right\Vert\right)\geq t\right\} +2e^{-B[et]\epsilon^{2}}. 
\end{equation} 
Using (\ref{extra1:1:1:1}) in the inequality above we obtain for $\epsilon,t\in\R^{>0}$ and $n\in\N:$
 \begin{equation}\label{eq2_29_may_2015}
 \begin{array}{rcl} 
  \mu\left\{ x:S^{\sigma}_{\tau_{n}}f\geq t\right\}\leq \left(1+ 2e^{B\epsilon^{3}}\right)\mu\left\{ x:\tau_{n}\cdot\left(\epsilon+\int fd\mu+\left\Vert f\right\Vert \right)\geq t\right\}.
  \end{array}
  \end{equation} 
 Applying logarithms to both sides in (\ref{eq2_29_may_2015}), dividing on both sides by $t\in \R^{>0},$ then taking $-\limsup_{t\to +\infty}$, and finally dividing both sides by $\mu(\sI_n)$ and letting $n$ tend to infinity, we conclude that
\begin{equation}\label{24072015_INEQ3}
\lim_{n\to\infty}\frac{\tilde{R}(\sI_n)}{\mu(\sI_{n})}\geq \frac{\gamma(z)+\epsilon\left\Vert f\right\Vert}{\epsilon+\int f d\mu+\left\Vert f\right\Vert}.
  \end{equation} 
Because $\epsilon\in \R^{>0}$ is arbitrary, we conclude the result.
  \end{proof}  
  
 We now complete the proof of some identities and inequalities used in the proof of (\ref{eq_enough}).

\begin{proof}[Proof of (\ref{24072015_ID1})]
We prove the statement:
$$
 \begin{aligned} 
\mu\left\{ x:S^{\sigma}_{\tau_{n}}f\geq t\right\}&= \mu\left\{ x:S^{\sigma}_{\tau_{n}}f\geq t,\tau_{n}>\epsilon t, \left|\frac{1}{[\epsilon t]}S^{\sigma}_{[\epsilon t]}f-\int fd\mu\right|<\epsilon\right\}\cr
&+\mu\left\{ x:S^{\sigma}_{\tau_{n}}f\geq t, \tau_{n}>\epsilon t, \left|\frac{1}{[\epsilon t]}S^{\sigma}_{[\epsilon t]}f-\int fd\mu\right|\geq\epsilon\right\}.
 \end{aligned} 
$$
 In fact 
 \[
 \mu\left\{ x:S^{\sigma}_{\tau_{n}}f\geq t\right\}=\mu\left\{ x:S^{\sigma}_{\tau_{n}}f\geq t,\tau_{n}\leq\epsilon t\right\}+\mu\left\{ x:S^{\sigma}_{\tau_{n}}f\geq t,\tau_{n}>\epsilon t\right\},
 \]
 but 
 \[
 \mu\left\{ x:S^{\sigma}_{\tau_{n}}f\geq t,\tau_{n}\leq\epsilon t\right\}\leq \mu\left\{ x: \epsilon t\left\Vert f\right\Vert \geq t\right\}=0 
 \]
 because $0<\epsilon<1/\left\Vert f\right\Vert.$
\end{proof}

\begin{proof}[Proof of (\ref{24072015_INEQ2})]
It is enough to prove the following inequality
$$
 \begin{aligned} 
&\mu\left\{ x:S^{\sigma}_{\tau_{n}}f\geq t,\tau_{n}>\epsilon t, \left|\frac{1}{[\epsilon t]}S^{\sigma}_{[\epsilon t]}f-\int fd\mu\right|<\epsilon\right\}\cr
&+\mu\left\{ x:S^{\sigma}_{\tau_{n}}f\geq t, \tau_{n}>\epsilon t, \left|\frac{1}{[\epsilon t]}S^{\sigma}_{[\epsilon t]}f-\int fd\mu\right|\geq\epsilon\right\} \cr
&\leq \mu\left\{ x:\tau_{n}\left(\epsilon+\int fd\mu\right)\geq t-\sum_{k=[\epsilon t ]}^{\tau_{n}-1}f\circ \sigma^k(x)\right\}+2e^{-B[et]\epsilon^{2}}.
 \end{aligned} 
$$
It involves two inequalities:
\begin{enumerate}
\item the first is 
$$
 \begin{aligned} 
&\mu\left\{ x:S^{\sigma}_{\tau_{n}}f\geq t,\tau_{n}>\epsilon t, \left|\frac{1}{[\epsilon t]}S^{\sigma}_{[\epsilon t]}f-\int fd\mu\right|<\epsilon\right\}\cr
&\leq \mu\left\{ x:\tau_{n}\left(\epsilon+\int fd\mu\right)\geq t-\sum_{k=[\epsilon t ]}^{\tau_{n}-1}f\circ \sigma^k(x)\right\}
 \end{aligned} 
$$
that comes from 
$$
 \begin{aligned} 
&\mu\left\{ x:S^{\sigma}_{\tau_{n}}f\geq t,\tau_{n}>\epsilon t, \left| \frac{1}{[\epsilon t]}S^{\sigma}_{[\epsilon t]}f-\int fd\mu\right|<\epsilon\right\}\cr
&\leq \mu\left\{ x:S^{\sigma}_{\tau_{n}}f\geq t,\tau_{n}>\epsilon t, \frac{1}{[\epsilon t]}S^{\sigma}_{[\epsilon t]}f\leq \int fd\mu+\epsilon  \right\}\cr
&\leq \mu\left\{ x: \frac{t-\sum_{k=[\epsilon t ]}^{\tau_{n}}f\circ \sigma^k(x)}{\tau_{n}} \leq \frac{S^{\sigma}_{\tau_{n}}f-\sum_{k=[\epsilon t ]}^{\tau_{n}}f\circ \sigma^k(x)}{\tau_{n}}\leq \int fd\mu+\epsilon \right\}\cr
&\leq\mu\left\{ x:\tau_{n}\left(\epsilon+\int fd\mu\right)\geq t-\sum_{k=[\epsilon t ]}^{\tau_{n}-1}f\circ \sigma^k(x)\right\},
 \end{aligned} 
$$
\item  the second is \[\mu\left\{ x:S^{\sigma}_{\tau_{n}}f\geq t, \tau_{n}>\epsilon t, \left|\frac{1}{[\epsilon t]}S^{\sigma}_{[\epsilon t]}f-\int fd\mu\right|\geq\epsilon\right\}\leq 2e^{-B[et]\epsilon^{2}}\]
 that comes from 
 $$
 \begin{aligned} 
&\mu\left\{ x:S^{\sigma}_{\tau_{n}}f\geq t, \tau_{n}>\epsilon t, \left|\frac{1}{[\epsilon t]}S^{\sigma}_{[\epsilon t]}f-\int fd\mu\right|\geq\epsilon\right\}\cr
&\leq \mu\left\{ x:\left|\frac{1}{[\epsilon t]}S^{\sigma}_{[\epsilon t]}f-\int fd\mu\right|\geq\epsilon\right\}\leq 2e^{-B[et]\epsilon^{2}}.
 \end{aligned} 
$$
\end{enumerate}
We can apply Theorem \ref{cm2011} and this concludes the proof.
\end{proof}
 
%23 SEP 13:51 
\begin{proof}[Proof of (\ref{eq2_29_may_2015})]
We have the following inequalities:
$$
 \begin{aligned} 
& \mu\left\{ x:S^{\sigma}_{\tau_{n}}f\geq t\right\}\cr
&\leq \mu\left\{ x:\tau_{n}\left(\epsilon+\int fd\mu\right)\geq t-\sum_{k=[\epsilon t ]}^{\tau_{n}-1}f\circ \sigma^k(x)\right\}+2e^{-B[et]\epsilon^{2}}\cr
&\leq \mu\left\{ x:\tau_{n}\left(\epsilon+\int fd\mu\right)\geq t-(\tau_{n}-[\epsilon t])\left\Vert f\right\Vert \right\}+2e^{-B[et]\epsilon^{2}}\cr
&\leq \mu\left\{ x:\tau_{n}\left(\epsilon+\int fd\mu+\left\Vert f\right\Vert \right)\geq t(1+\epsilon \left\Vert f\right\Vert )\right\}+2e^{-B[et]\epsilon^{2}}\cr
&\leq \mu\left\{ x:\tau_{n}\left(\epsilon+\int fd\mu+\left\Vert f\right\Vert \right)\geq t\right\}+ 2e^{B\epsilon^{3}}\mu\left\{ x:\tau_{n}\left(\epsilon+\int fd\mu\right)\geq t\right\}\cr
&\leq \left(1+ 2e^{B\epsilon^{3}}\right)\mu\left\{ x:\tau_{n}\left(\epsilon+\int fd\mu+\left\Vert f\right\Vert \right)\geq t\right\}. 
 \end{aligned} 
$$
\end{proof}

\begin{proof}[Proof of (\ref{24072015_INEQ3})]
Inequality (\ref{eq2_29_may_2015}) implies
$$
 \begin{aligned} 
& \limsup_{t\to +\infty}\frac{1}{t}\log \mu\left\{ x:S^{\sigma}_{\tau_{n}}f\geq t\right\}\cr
&\leq \limsup_{t\to +\infty}\frac{1}{t}\log\mu\left\{ x:\tau_{n}\left(\epsilon+\int fd\mu+\left\Vert f\right\Vert \right)\geq t\right\}. 
 \end{aligned} 
$$
 Finally, we can write an inequality that does not depend on $\epsilon\in\R^{>0}$ that concludes the result:
 $$
 \begin{aligned} 
 &\lim_{n\to\infty}-\frac{1}{\mu(\sI_n)}\limsup_{t\to +\infty}\frac{1}{t}\log \mu\left\{ x:S^{\sigma}_{\tau_{n}}f\geq t\right\}\cr
&\geq \lim_{n\to\infty}-\frac{1}{\mu(\sI_n)}\limsup_{t\to +\infty}\frac{1}{t}\log\mu\left\{ x:\tau_{n}\left(\epsilon+\int fd\mu+\left\Vert f\right\Vert \right)\geq t\right\}\cr
&=\frac{\gamma(z)}{\epsilon+\int fd\mu+\left\Vert f\right\Vert },
 \end{aligned} 
 $$
 where we used Theorem \ref{teoFP} in the last equality.
\end{proof}

\cleardoublepage

\end{document}